\documentclass[12pt]{amsart}

\usepackage{amsmath,amsthm,amssymb}
\usepackage{bbm}
\usepackage{mathrsfs}
\usepackage{enumerate}
\usepackage{graphicx}
\usepackage{graphics}
\usepackage{caption}
\usepackage{subcaption}
\usepackage{thmtools}
\usepackage{thm-restate}
\usepackage{hyperref}
\usepackage{float}
\usepackage{color}
\usepackage[normalem]{ulem}

\newtheorem{thm}{Theorem}
\newtheorem{conj}[thm]{Conjecture}

\newtheorem{corr}[thm]{Corollary}
\newtheorem{lemma}[thm]{Lemma}
\newtheorem{question}[thm]{Question}
\theoremstyle{definition}

\newtheorem{defn}[thm]{Definition}

\theoremstyle{remark}
\newtheorem{rem}[thm]{Remark}

\begin{document}

\title{NP--hard problems naturally arising in knot theory}
\author{Dale Koenig and Anastasiia Tsvietkova}
\date{}
 \footnotesize
 \begin{abstract}
We prove that certain problems naturally arising in knot theory are NP--hard or NP--complete.
These are the problems of obtaining one diagram from another one of a link in a bounded number of Reidemeister moves, determining whether a link has an unlinking or splitting number $k$, finding a $k$-component unlink as a sublink, and finding a $k$-component alternating sublink.
\end{abstract}
\maketitle
\normalsize
\section{Overview}

Many problems that lie at the heart of classical knot theory can be formulated as decision problems, with an algorithm being a solution.
Among them, for example, is the question of equivalence (up to isotopy) of two links given by their diagrams.
This can be approached in many different ways, among which applying Reidemeister moves perhaps has the longest history.
Other examples are the unknotting, unlinking and splitting number questions, i.e.
arriving to a diagram of an unlink, an unknot, or a split diagram from some diagram by interchanging overpasses and underpasses in a certain number of crossings.

The complexities of these basic decision problems in knot theory are not yet well-understood.
Despite the lack of polynomial algorithms, few problems in knot theory are known to be NP--hard or NP--complete at all.
Those few problems are the problem of determining a bound on the genus of a knot in a general 3-manifold \cite{AgolHassThurston1} improved to knots in $S^3$ in \cite{Lackenby1}, the problem of detecting a sublink isotopic to a given link, and the problem of determining a bound for the Thurston complexity of a link  \cite{Lackenby1}.
The purpose of this paper to establish that a number of other natural problems in knot theory are NP--hard.
The problems described here have a advantage that they can be formulated solely in terms of link diagrams.
They are described below.
For an overview of complexity theory including NP--hardness and NP--completeness, see, for example, \cite{GareyJohnson1}.
The status of many decision and complexity problems from knot theory is discussed in \cite{Lackenby3}.

In section \ref{sec_unlinksublink} we look at the \textsc{unlink as a sublink} problem, a special case of the \textsc{sublink problem} defined and proven to be NP--hard by Lackenby \cite{Lackenby1}.
We prove that the \textsc{unlink as a sublink} problem is NP--hard.
Note that restricting a problem makes it easier, so the proof that the restricted problem is hard also implies that the general problem is hard.
Thus, NP--hardness of the \textsc{sublink problem} is a corollary of the NP--hardness of the problem that we consider.
The techniques used in the proof are widely used throughout this paper.
More generally, for any property $X$ of links one can also consider decision problems of the form "Given a diagram of a link L and a positive integer k, is there a k component sublink of $L$ with the property $X$?"  We show in section \ref{sec_alternatingsublink} that this is also NP--hard if $X$ is the property of being an alternating link.
We expect that many other problems of this form are NP--hard.

Very little is known about unknotting, unlinking and splitting numbers in general, without restricting to particular classes of links.
It is possible that these invariants are not even computable.
Moreover, it is not known whether there is an algorithm to detect whether a knot has unknotting number 1.
A recent breakthrough by Lackenby suggests an algorithm to determine whether a hyperbolic link satisfying certain restrictions has unlinking or splitting number one \cite{Lackenby5}, but no algorithm for the general case is yet known.
We provide the first lower bounds on the complexities of the general unlinking and splitting number problems in sections \ref{sec_unlinkingnum} and \ref{sec_splittingnum}.

Reidemeister showed that any two diagrams of the same link can be taken one to another by a sequence of Reidemeister moves \cite{Reidemeister1}.
The bounds for the number of moves have been a long-standing question, and increasingly good results have been obtained for a general case, as well as some special cases (see, for example, \cite{HassLagariasPippenger1,Lackenby2,Lackenby4}).
Such bounds are especially powerful since such a bound for a given link type gives an algorithm to detect that link type, and the existence of a polynomial bound ensures that the detection is in NP.
For an arbitrary link diagram $D$, a bound exponential in the number of crossings of $D$ has been proven to exist by \cite{CowardLackenby1}, so there exists an algorithm that can determine whether two diagrams are related by Reidemeister moves (and therefore whether the links are equivalent).
In section \ref{sec_reidbound} we provide the first lower bound on the complexity of this problem by proving that the decision problem of determining whether two diagrams are related by at most $k$ Reidemeister moves is NP--hard.

Every section that follows is devoted to one of the above problems.
While the most interesting results are perhaps the ones that concern Reidemeister moves and unlinking number, we put the sections in the order that makes it easier for the reader to follow the proofs, since some of the arguments can be seen as refinements of the others.

\section{The unlink as a sublink problem}
\label{sec_unlinksublink}

The \textsc{sublink problem} asks "Given diagrams of two links, is there a sublink of the first that is isotopic to the second?"  Lackenby showed that this problem is NP--hard using a Karp reduction (see Section 3.1.2 in \cite{Goldreich}) from the \textsc{hamiltonian path} problem \cite{Lackenby1}.
Here we examine the \textsc{unlink as a sublink} problem, in which the second link is an unlink.

\vspace{0.14in}

\textsc{unlink as a sublink}: given a diagram for a link $L$ and a positive integer $k$, is there a $k$-component sublink of $L$ that is an unlink?

\begin{thm}
\textsc{unlink as a sublink} is NP--complete.
\label{thm_unlinksublink}
\end{thm}
\begin{proof}
We first prove NP--hardness by providing a Karp-reduction of the \textsc{3--SAT} problem to the given problem.
Suppose there are variables $x_1 \dots x_n$ and clauses $c_1\dots c_m$, in a formula $F$, where each clause is of the form
\begin{center}$z_{\alpha_i} \vee z_{\beta_i} \vee z_{\gamma_i}$, $1 \le \alpha_i < \beta_i < \gamma_i \le n$,
\end{center}

\noindent and $z_k$ represents either $x_k$ or $\neg x_k$.
We construct a $(2n+m)$-component link $L^F$ such that there exists an $(n+m)$-component unlink as a sublink if and only if there is a variable assignment to $x_1 \dots x_m$ satisfying all the clauses $c_1 \dots c_m$.

To begin, create $n$ Hopf links, one for each variable $x_\alpha$.
See Figure \ref{fig_hopf}.
Denote this link by $L_0^F$.

\begin{figure}[ht]
\centering
\includegraphics[height=2.1in]{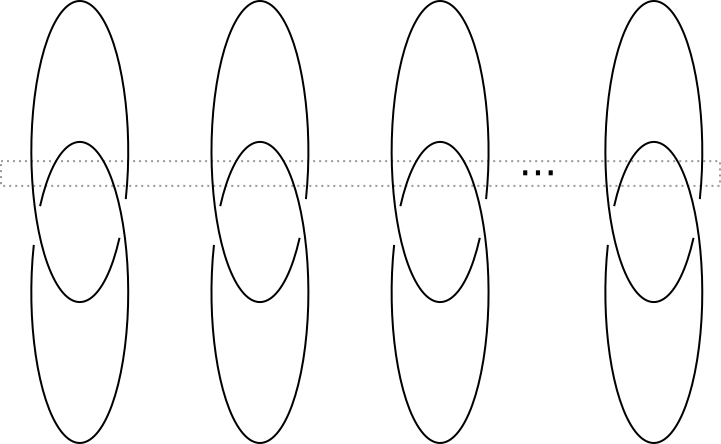}
\caption{Hopf links, one for each variable $x_i$.}
\label{fig_hopf}
\end{figure}

Label the top components of the Hopf links $x_1 \dots x_n$ and the bottom components $\neg x_1 \dots \neg x_n$.
Zoom in on the grey dotted box in Figure \ref{fig_hopf} and divide it vertically into $m$ shorter boxes as in Figure \ref{fig_relationarea}.
For each clause $c_i$, we will add one more link component which we call a clause component lying in the $i$th grey dotted box.
In this box we see two strands of the link $L_0^F$ labelled $x_\alpha$ and two strands labelled $\neg x_\alpha$ for each $\alpha$.
Pick a single strand for each label, and leave the remaining strands unlabeled.

\begin{figure}[ht]
\centering
\includegraphics[height=1.8in]{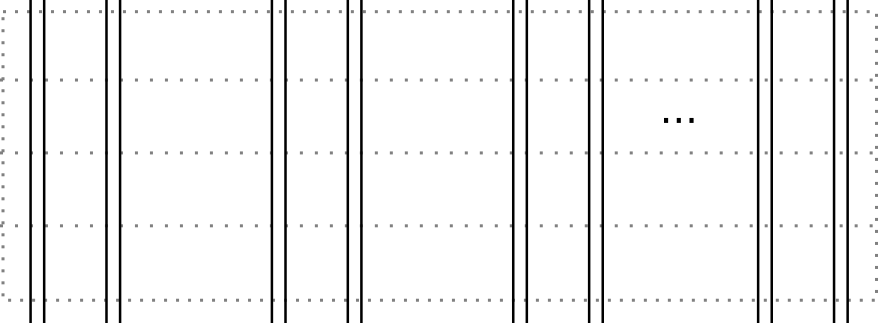}
\caption{Dividing the box into m shorter boxes vertically, to put each clause component in.
}
\label{fig_relationarea}
\end{figure}

Consider the Wirtinger presentation of the link group $\pi_1(S^3-L^F_0)$. A word in the generators corresponds to a string of arcs/arrows
passing under the labeled strands, passing left to right for $x$ and $\neg x$, and right to left for $x^{-1}$
and ${\neg x}^{-1}$. By placing such arrows from top to bottom in the diagram, and stringing them together
using arcs which pass over the strands in the diagram, we can introduce a loop representing any word
which has no self-crossings. This construction is suggested in, for example, \cite{Rolfsen1}, and Figure \ref{fig_clauselink} gives an example. In this way, draw a loop corresponding to the word $[z_{\alpha_i}, [ z_{\beta_i}, z_{\gamma_i}]]$ (an iterated commutator).
Let the clause component corresponding to $c_i = z_{\alpha_i} \vee z_{\beta_i} \vee z_{\gamma_i}$  be this loop. For the formula $F$, we then obtain the link $L^F$ by adding all clause components to $L_0^F$.

Each clause component links with three variable components in a Brunnian way.
Indeed, an iterated commutator is trivial as a group element if any of the commuting elements is trivial, and the associated word can be reduced to the empty word by cancelling adjacent generators and inverses.
Consider a clause $c_i = z_{\alpha_i} \vee z_{\beta_i} \vee z_{\gamma_i}$ and assume we have erased one of $z_{\alpha_i}$, $z_{\beta_i}$, or $z_{\gamma_i}$, so the iterated commutator can be reduced to the empty word by such cancellations.
With the way we have drawn the clause component, each cancellation in the commutator gives a sequence of type II Reidemeister moves simplifying the diagram.
As an illustration, in our example in Figure \ref{fig_clauselink}, after removing one of $\neg x_1, \neg x_3$, or $x_4$, the clause component can be pulled away from the other components to become an unlinked unknot.

It remains only to show that the clause component links the three variable components in a nontrivial way.
This follows from the fact that the fundamental group of the complement of the three variable components is free, and the clause component represents a nontrivial group element in that free group. It is therefore not
freely homotopic to a trivial loop, since free homotopy corresponds to conjugation in the group, and
so cannot be isotopic to a trivial loop.

\begin{figure}[ht]
\centering
\includegraphics[height=3in]{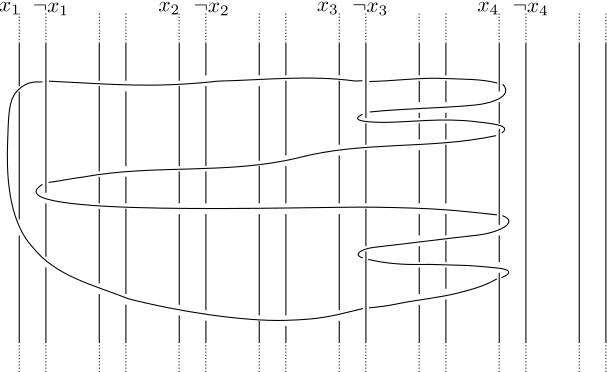}
\caption{A link representing the commutator $[\neg x_1,[\neg x_3,x_4]]$, corresponding to the clause $ \neg x_1 \vee \neg x_3 \vee x_4$.}
\label{fig_clauselink}
\end{figure}

Each clause component has precisely ten undercrossings, and has at most $8n$ crossings between two such undercrossings.
Indeed, between undercrossings it passes over each of the $4n$ strands in Figure \ref{fig_relationarea}  at most twice (in fact, at most once except when returning from the bottom back to the top).
It follows that the number of crossings introduced for each new clause component is at most $64n + 10$.
There are $m$ clauses and $2n$ crossings corresponding to the $n$ Hopf links, so the final diagram has at most $2n+m(64n+10)$ crossings, which is polynomial in the size of the input formula $F$.

We now argue that the link $L^F$ contains an $(n+m)$-component unlink as a sublink if and only if the corresponding formula $F$ is satisfiable.

Suppose there is some set of variable assignments making $F$ true.
If the variable $x_\alpha$ is assigned to be TRUE, then delete the component labelled $x_\alpha$ from the link $L^F$, and if it is assigned FALSE, delete the component labelled $\neg x_\alpha$.
For each clause component $c$ of $L^F$ that corresponds to a commutator $[z_{\alpha_i}, [ z_{\beta_i}, z_{\gamma_i}]]$, one of $z_{\alpha_i}$, $z_{\beta_i}$, or $z_{\gamma_i}$ is TRUE, and the link component with the respective label was deleted.
Hence the clause component $c$ is unlinked from the rest of the diagram of $L$ by the Brunnian property.
After pulling away all the $m$ clause components, we are left with the $n$ unknotted components (``halves") of the Hopf links, so we have an $(n+m)$-component unlink.

Conversely, assume the link $L^F$ contains an $(n+m)$-component unlink $U$ as a sublink.
We want to prove that $F$ is satisfiable.
The sublink $U$ must contain at most one component from each of $n$ Hopf links, since otherwise it contains pairs of components with nonzero linking number.
$U$ must therefore contain every clause component and exactly one component of each of the $n$ Hopf links in order to have $n+m$ components total.
This corresponds to the variable assignment determined by deleting the component labelled either by $x_\alpha$ or  by $\neg x_\alpha$ for each $\alpha$.
If any clause of the formula $F$ is not satisfied as a result of this assignment, the corresponding clause component will form a non-trivial Brunnian link with the three respective ``halves" of Hopf links, and thus $U$ is not an unlink, a contradiction.

We have therefore reduced the 3--\textsc{SAT} problem to the \textsc{unlink as a sublink} problem.
The number of crossings of the diagram of $L^F$ is bounded by a polynomial in the size of $F$.
This implies we have a Karp reduction, so \textsc{unlink as a sublink} is NP--hard.

To see that the problem is NP, we observe that a polynomial length certificate consists of a choice of $k$ link components followed by a sequence of Reidemeister moves (and resulting diagrams) converting the corresponding $k$-component link diagram to the trivial diagram of a $k$-component sublink.
The number of Reidemeister moves needed to split the diagram is polynomial as per \cite{HassLagariasPippenger1, Lackenby2}, and this splitting needs to be performed up to $k-1$ times.
The number of moves needed to show each component is an unknot is polynomial by \cite{HassLagariasPippenger1, Lackenby2} as well.
Moreover, the size of the diagram at each step is bounded by a polynomial in the size of the original diagram, since the number of crossings is fewer than the original number of crossings plus the number of Reidemeister moves performed.
Thus there is a polynomial length certificate showing that the original link contains a $k$-component unlink as a sublink, and \textsc{unlink as a sublink} is in NP.
\end{proof}
Note that NP--hardness of the \textsc{sublink problem} as proven in \cite{Lackenby1} follows immediately as a corollary, since every instance of \textsc{unlink as a sublink} is also an instance of the \textsc{sublink problem}.

\vspace*{1.5cm}
\section{Unlinking number}\label{sec_unlinkingnum}

We next consider the problem of calculating the unlinking number of a link.
The \emph{unlinking number} of a link is the minimum number of crossing changes required to become an unlink, minimized over all diagrams of the link.
The definition can also be formulated without reference to knot diagrams as is described below.

\begin{defn}
Suppose $L$ is a link and $\alpha$ a simple arc with endpoints lying on $L$ and disjoint from $L$ otherwise.
An \emph{unlinking move} corresponding to the arc $\alpha$ is a homotopy of $L$ that is the identity outside of a regular neighborhood $N(\alpha)$ of $\alpha$, and within $N(\alpha)$ performs a move in which the link passes over itself once as in Figure \ref{fig_unlinkingmove}.\end{defn}

\begin{figure}[ht]
\centering
    \begin{subfigure}[b]{0.5\textwidth}
        \centering
        \includegraphics[height=1.5in]{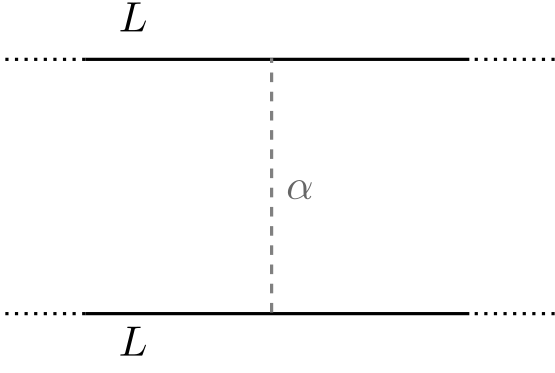}
        \caption{}
    \end{subfigure}%
~~~~~
    \begin{subfigure}[b]{0.5\textwidth}
        \centering
        \includegraphics[height=1.5in]{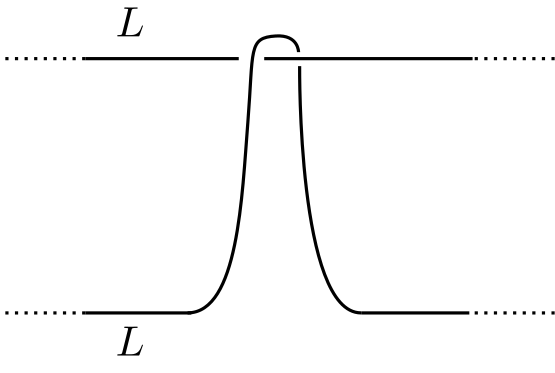}
        \caption{}
    \end{subfigure}

\caption{An example of an unlinking move replacing (\textsc{a}) with (\textsc{b}).}
\label{fig_unlinkingmove}
\end{figure}

An arc $\alpha$ alone does not uniquely determine a given unlinking move, as the link in Figure \ref{fig_unlinkingmove} (\textsc{b}) might twist around the arc.
Rather, there are infinitely many different unlinking moves that can be performed for any given arc.
However, for our purposes distinguishing these moves is not important;  we only need a way to represent where an unlinking move occurs.

A finite collection of unlinking moves can be represented by a collection of disjoint arcs with disjoint regular neighborhoods.
Then each unlinking move does not affect the neighborhoods of the arcs associated with the other unlinking moves.
Therefore, all unlinking moves in such a collection can be performed simultaneously or in any order without changing the resulting link.

With this definition, the unlinking number of a link can be defined as the smallest number of unlinking moves needed to change the link into the unlink. Since unlinking moves are defined by paths in the link complement, they are independent of the link diagram. Hence one can take the smallest number of unlinking moves for one given link diagram, unlike with crossing changes, where the number of changes is minimized across all link diagrams.

Nonetheless, the two definitions of unlinking number are equivalent. Any crossing change is an unlinking move corresponding to the vertical arc connecting the two points of the crossing.
For the other direction, given a collection of unlinking moves, first perform an isotopy of $L$ to make all of the unlinking move arcs perpendicular to a plane $\Pi$.
This can be done by first choosing a thin regular neighborhood of each arc, performing an isotopy that shrinks the arcs until they lie in small disjoint Euclidean balls, followed by an ambient isotopy that is the identity outside the set of balls to straighten the arcs.
These isotopies will also move the link $L$.
Once all arcs are straight and perpendicular to $\Pi$, project $L$ to $\Pi$.
After applying a small perturbation of $L$ away from the arcs if necessary, the result of this projection will be a link diagram by general position arguments. Then each unlinking move's arc is a crossing arc in the diagram.
However, it is not yet true that the unlinking moves correspond to the obvious crossing changes.  If the unlinking move twists about its arc, we must additionally perform an isotopy to remove this twisting.  Such an isotopy will introduce new crossings, and by slightly perturbing the diagram we can ensure again that all crossings are distinct.  At the end of this process, each unlinking move corresponds to a crossing change in the diagram as desired.

Thus, given a diagram of $L$ and a set of $n$ crossing changes turning $L$ into the unlink, there is a set of $n$ unlinking moves making $L$ into an unlink.
Conversely, if there is a set of $n$ unlinking moves making $L$ into an unlink then there is a diagram for which $n$ crossing changes make $L$ an unlink.
Therefore, the minimum number of crossing changes needed to make $L$ an unlink is the same as the minimum number of unlinking moves needed.

\vspace{0.14in}
\textsc{unlinking number}: Given a link diagram and an integer $n$, does the link  have unlinking number $n$?
\vspace{0.14in}

It is not known whether \textsc{unlinking number} is computable, and it is not expected to be in NP \cite{Lackenby3}.
Below, we prove that it is NP--hard (Theorem \ref{thm_unlinking}).
The difficulty with finding a polynomial certificate lies in the fact that crossing changes required to efficiently get an unlink may not be visible in a given diagram.
However one can consider the following restricted problem.

\vspace{0.14in}
\textsc{diagrammatic unlinking number}: Given a link diagram $D$ and an integer $n$,  can the link be made into an unlink by $n$ crossing changes in the diagram $D$?
\vspace{0.14in}

The restricted problem is NP--complete, which we will prove as a corollary to the proof of Theorem \ref{thm_unlinking}.

\begin{thm}
\textsc{unlinking number} is NP--hard.
\label{thm_unlinking}
\end{thm}

We will prove this by modifying our construction for the \textsc{unlink as a sublink problem}.
In $L^F$,  we replace each variable component with its untwisted Whitehead double.
Figure \ref{fig_hopfwhitehead} demonstrates the replacement for a single Hopf link.
We continue to refer to the Whitehead doubled variable components as variable components.

\begin{figure}[ht]
\centering
    \begin{subfigure}[b]{0.4\textwidth}
        \centering
        \includegraphics[height=1.5in]{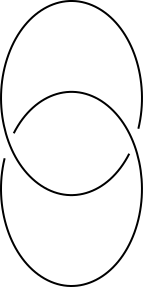}
        \caption{}
    \end{subfigure}%
~
    \begin{subfigure}[b]{0.4\textwidth}
        \centering
        \includegraphics[height=1.5in]{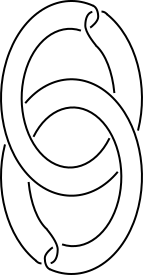}
        \caption{}
    \end{subfigure}

\caption{(\textsc{a}) A Hopf link and (\textsc{b}) the link obtained by replacing each component with a Whitehead double.}
\label{fig_hopfwhitehead}
\end{figure}

In the resulting link, any variable component can be made into an unknot unlinked from the remainder of the diagram using a single crossing change.
In fact, there are two such distinct crossing changes, but they determine only one unlinking move. This is the same unlinking move that is determined by any crossing change in a standard diagram of a Whitehead link.
We call this unlinking move an \emph{unclasping move}.
Using $n$ such unlinking moves followed by Reidemeister moves we can replace $n$ variable components from the connected diagram that we have with an $n$-component unlink, disjoint from the rest of the diagram.
Moreover, if there exists a set of $n$ unlinking moves resulting in the unlink, then Lemma \ref{lem_unlinking}(a) asserts that there exists another set of $n$ unlinking moves, also resulting in the unlink, such that any move involving a variable component is an unclasping move for that component.
The second part of the lemma (Lemma \ref{lem_unlinking}(b)) will be used in Section \ref{sec_splittingnum}.

Here and further, if $L$ is a link and $K$ is a component of $L$, denote by $L \backslash K$ the sublink of $L$ consisting of components other than $K$.

\begin{lemma}
\label{lem_unlinking}
Suppose $L$ is a link, and $K$ is a component of $L$.
Let $C=c_1, c_2, ..., c_n$ be a set of $n$ unlinking moves of $L$, at least one of which involves the component $K$.
Let $c$ be an unlinking move of $L$ that involves the component $K$ and results in $K$ becoming an unknot split from the other components of $L$.
Then we can replace $C$ with another sequence $C'$ of unlinking moves beginning with $c$ such that
\vspace{0.1in}

(a) $C'$ has length equal or less than that of $C$.

(b) If $C$ results in $L$ being an unlink, so does $C'$.
\vspace{0.1in}

(c) If $P$ is some component of $L \backslash K$ and $C$ results in $P$ being split from $L \backslash P$, so does $C'$.
\end{lemma}

\begin{proof}
Since the first unlinking move of $C'$ is $c$, the component $K$ is unknotted and unlinked from $L\backslash K$ after $c$ is applied, and $L\backslash K$ is unchanged.
Let $D$ be the disk bounded by $K$ after $c$ is applied.
We now construct $C'$ iteratively by modifying moves from $C$.

For each move of $C$, if the move involves the component $K$, we do not include the move in $C'$.
We will obtain the remaining moves of $C'$ by modifying the moves of $K$ that do not involve $K$.
Since at least one move of $C$ is assumed to involve $K$, this ensures that condition (a) is upheld.
If a move $c_i$ does not involve $K$, and the path determining $c_i$ does not intersect the disk $D$, we add $c_i$ to $C'$ unmodified.
If the path does pass through $D$, we isotope the path relative its boundary in the complement of $L \backslash K$ to make it disjoint from $D$. Figure \ref{ModifyMove} shows a link with a component $K$ that is unlinked and unknotted. In the diagram on the left, the green arc corresponds to an unmodified unlinking move (a crossing change), and passes through the disk $D$ bounded by $K$. In the diagram on the right, the green arc is isotoped in the complement of $L \backslash K$ (but not in the complement of $L$) to correspond to a modified unlinking move that avoids $D$. This modification is always possible  when $K$ is unlinked from $L\backslash K$.
Note that since we isotope in the complement of $L\backslash K$, the path must intersect $K$ at some point in time during the isotopy.
We add the modified unlinking move to $C'$, and repeat for each unlinking move $c_i$ of $C$

\begin{figure}[ht]
\centering

        \includegraphics[height=1.6in]{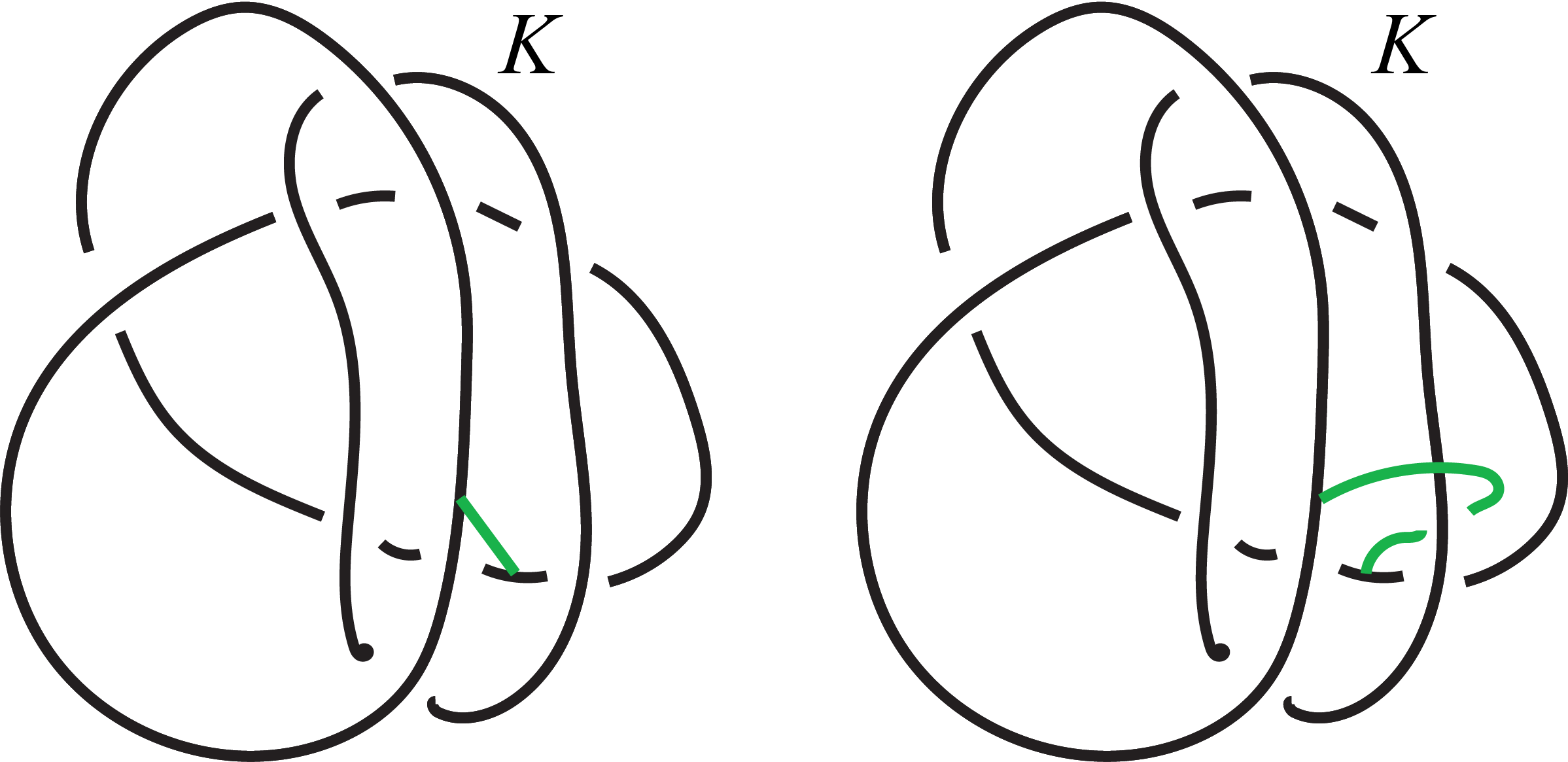}

\caption{Modifying an unlinking move from $C$.}
\label{ModifyMove}
\end{figure}

Let $C' = c, c_1', ..., c_m'$, where for notational convenience we do not assume that $c_i'$ is the result of modifying $c_i$ (for example, if $c_1$ involved $K$ then we would have skipped it, so $c_1'$ would be a result of modifying $c_2$).
Since the paths determining the modified moves were constructed to be disjoint from $D$, the disk $D$ is disjoint from $L\backslash K$ after performing $C'$.
Therefore, after performing all unlinking moves of $C'$, $K$ must be an unlinked unknot.
To evaluate conditions (b) and (c), we therefore only need to analyze how the moves of $C'$ affect $L \backslash K$.
If we restrict both $C$ and $C'$ to the subset of moves that do not involve $K$, then the sets of restricted moves act similarly on $L\backslash K$, since the paths determining the moves only differ by an ambient isotopy of $L \backslash K$. So (b) and (c) hold.

\end{proof}

We can now prove Theorem \ref{thm_unlinking}.

\begin{proof}[Proof of Theorem \ref{thm_unlinking}]
For a given 3--\textsc{SAT} problem $F$ with $n$ variables and $m$ clauses, take the associated link $L^F$ constructed for \textsc{unlink as a sublink}, and modify it by replacing every component with an untwisted Whitehead double.
Denote this link by $L_{*}^F$.
In this diagram, unclasping moves for each Whitehead doubled component are visible as crossing changes in the diagram. Each Whitehead doubled component bounds an obvious disk in the diagram after performing such an unclasping move.

Suppose there is a sequence $C$ of $n$ unlinking moves of $L_{*}^F$ resulting in an unlink.
A Whitehead doubled Hopf link cannot be unlinked without an unlinking move that only involves the components of that Whitehead doubled Hopf link. Indeed, moves involving other components will not split the two components of a Whitehead doubled Hopf link from each other, and a crossing change between the
two components will give them non-zero linking number. Note however that once a component of Whitehead doubled Hopf link is split from all other components, it is an unlinked unknot. Denote the Whitehead doubled Hopf sublinks by $W_1, W_2, ..., W_n$. Since there are $n$ distinct Whitehead doubled Hopf sublinks of $L_{*}^F$ and $C$ consists of $n$ unlinking moves, $C$ must consist of one move involving each of $W_i$. Note that such sequence $C$ is not unique: we just choose one such $C$.

We claim that by repeatedly applying Lemma \ref{lem_unlinking} to $L_{*}^F$ and one of the components of each $W_i$, we obtain another set of unlinking moves $C_*$ such that the $i$th move of $C_*$ is an unclasping move for either the $x_i$ or $\neg x_i$ component of $L_*^F$. Indeed, first find an unlinking move of $C$ involving one of the components $K_1$ of $W_1$.
Applying Lemma \ref{lem_unlinking} to $L$ and $K_1$ gives a sequence of unlinking moves $C'$ of equal or lesser length than $C$.
The first unlinking move $c_1'$ of $C'$ is an unclasping move for $K_1$.
By Lemma \ref{lem_unlinking}(b), the new sequence of moves $C'$ must still result in $L$ being an unlink.
Note that $C'$ must have a move involving one of the components $K_2$ of $W_2$, since $C$ had a move involving only $W_2$, and in the proof of Lemma \ref{lem_unlinking} we did not exclude such moves.
So we can again apply Lemma \ref{lem_unlinking} to $L$ and the component $K_2$ of $W_2$.
Recall that we are working with a particular link diagram of $L_*^F$, described in the beginning.
After performing an unclasping move $c_2'$ for $K_2$, the component $K_2$ bounds a disk visible in the link diagram, and the path determining $c_1'$ does not intersect this disk. Therefore, the application of Lemma \ref{lem_unlinking} does not modify $c_1'$, and we obtain a sequence beginning with $c_2', c_1'$.
Repeating this process, we can get a sequence of unlinking moves such that, up to reordering, its $i$th unlinking move is an unclasping move for either the $x_i$ or $\neg x_i$ component of $L_*^F$.

Deleting an unlinked unknot does not change whether a link is an unlink. So every unlinking move of $C_*$ can be diagrammatically seen as deleting a component of $L_*^F$, as in the proof of NP--hardness for \textsc{unlink as a sublink} (Theorem \ref{thm_unlinksublink}).

We now prove that 3--\textsc{SAT} formula $F$ is satisfied if and only if the link $L^F_*$ has unlinking number $n$.

As in the proof of Theorem \ref{thm_unlinksublink}, when formula $F$ is satisfied, we unclasp the variable components of $L^F_*$ labelled by the variables with values 1. Then the clause components can be isotoped away, and together with the final Whitehead doubles they form an unlink.

For the other direction,  suppose the unlinking number of $L^F_*$ is $n$. Above we proved that the set of $n$ unlinking moves $C$ involves each of $W_i$. Moreover, we proved above that we can substitute the sequence $C$ by the sequence $C*$ of the moves such that the $i$th move of $C_*$ is an unclasping move for either the $x_i$ or $\neg x_i$ component of $L_*^F$. We assign TRUE to the variable of $F$ for which the component of $L_F^*$ was unclasped. For a clause $c$ of $F$, consider the respective clause component of $L^F_*$ and the variable link components (Whitehead doubles) corresponding to the variables of $c$. They form a 4-component sublink $L_U$ of $L^F_*$. If $L_U$ is not an unlink, one of the variable components it includes is unclasped in $C_*$, and the clause $c$ is satisfied. But $L_U$ is a satellite link of a Brunnian link, and Brunnian link is not an unlink.  Hence $L_U$ contains an incompressible non-boundary parallel torus in its complement in $S^3$, and therefore cannot be an unlink as well.

Therefore a set of unlinking moves $C_*$ that results in $L_*^F$ becoming an unlink corresponds to a variable assignment satisfying the 3--\textsc{SAT} problem $F$.

The number of crossings in $L_*^F$ is at most four times the number of crossings in $L^F$ plus twice the number of components (the number of components is bounded by the number of crossings of $L^F$).
In the proof of Theorem \ref{thm_unlinksublink} we show that the number of crossings of $L^F$ is polynomial in the size of $F$, therefore so is the number of crossings of $L^F_*$. Hence the reduction above is polynomial in the size of $F$.
\end{proof}

\begin{corr}
\textsc{diagrammatic unlinking number} is NP--complete.
\end{corr}
\begin{proof} In this proof, for every 3SAT formula $F$ we need to construct a link diagram (rather than a link) that can be unlinked with $n$ unlinking moves if and only if $F$ is satisfiable. We will use the initial diagram of $L_*^F$ from above. With this, the fact that \textsc{diagrammatic unlinking number} is NP--hard follows from the proof of Theorem \ref{thm_unlinking}.
In the proof, a set of $n$ unlinking moves $C$ assumed to take $L_*^F$ to the unlink is modified to obtain a new set $C_*$ that also changes $L_*^F$ into an unlink.
In particular, all moves of $C_*$ consist of changing a crossing of a Whitehead double of an unknot.
All such moves are actually realized in the diagram, so there exists a set of $n$ crossing changes changing the diagram of $L_*^F$ into a diagram of the unlink if and only if the formula $F$ has a solution.
Thus, the restricted problem is NP--hard.

For proving that the problem is in NP, note that $L_*^F$ can be changed to a diagram of an unlink with $n$ crossing changes, though the diagram might not yet be the standard unlink diagram consisting of $n$ circles with no crossings.
Hence, we need to take a polynomial length certificate consisting of a choice of these $n$ crossings followed by one of the known polynomial length certificates for the unlink.
This proves the problem is NP--complete.
\end{proof}

\begin{rem}
  An alternative construction for the link $L^F_*$ is to take the link $L^F$ and replace only the Hopf link components with the Whitehead doubles.

\end{rem}

While we do not use this directly, we note that Theorem 1.2 from \cite{CowardLackenby2} shows that there is a unique way, up to a certain natural notion of equivalence, to unknot by a crossing change a Whitehead double of any knot but the figure-eight.

\section{Splitting and Unknotting numbers}
\label{sec_splittingnum}

There are two other natural problems related to \textsc{unlinking number}.
The first is the unknotting number problem, which is the restriction of the unlinking number problem to single component links.

\vspace{0.14in}
\textsc{unknotting number}: Given a knot diagram and an integer $n$, does the knot shown have unknotting number $n$?
\vspace{0.14in}

The next is the problem of calculating the splitting number of a link.
\begin{defn}
The \emph{splitting number} of a link $L$ is the minimal number of crossing changes required to make the link a split link, minimized over all diagrams.
\end{defn}
As with the unlinking number, the splitting number can also be formulated without reference to diagrams by using unlinking moves.
Specifically, it is the minimum number of unlinking moves needed to change a link into a split link.
Variations of the splitting number have also been analyzed in the literature (see, for example, \cite{Lackenby3} for details).
For example, while the splitting number only requires one separating 2-sphere in the link complement,  the \emph{total splitting number} asks for the minimum number of crossing changes needed to ensure that every component is split by a 2-sphere from the other components.

\vspace{0.14in}
\textsc{splitting number}: Given a link diagram and an integer $n$, does the link have splitting number $n$?
\vspace{0.14in}

The \textsc{unknotting number} and \textsc{splitting number} problems can also be rephrased as upper bound problems with ``number at most $n$''.
Although rephrasing in this way changes the problems, it does not affect the property of being NP--hard.
Indeed, $n$ calls to an oracle solving the exact value problem is sufficient to solve the upper bound problem, while two calls to an oracle for the upper bound problem will solve the exact value problem.
Since both the crossing number and splitting number are bounded by the crossing number and therefore the size of the input, both reductions are polynomial.

If one tries to extend the proof of Theorem \ref{thm_unlinking} to the unknotting number, the main difficulty is that an unknotting number is known only for few specific classes of knots.
Therefore, when constructing a new knot corresponding to a given 3--\textsc{SAT} instance, it is hard to show that there are no unexpected ways of unknotting it.
Nonetheless, we conjecture:

\begin{conj} \textsc{unknotting number} is NP--hard.
\end{conj}

We will use methods similar to the ones in Theorem \ref{thm_unlinking} to prove the following theorem.

\begin{thm}
\textsc{splitting number} is NP--hard.
\label{thm_splitting}
\end{thm}

 We start with two preliminary lemmas.

\begin{lemma}
Let $x_1, \cdots , x_n$ be generators for the free group $F_n$ and let $g$ be a nonempty product of nontrivial twice iterated commutators in $x_1, \cdots, x_n$, so $g = [x_{\alpha_1}, [ x_{\beta_1}, x_{\gamma_1}]][x_{\alpha_2}, [ x_{\beta_2}, x_{\gamma_2}]] \cdots  [x_{\alpha_m}, [ x_{\beta_m}, x_{\gamma_m}]]$, where $1 \le \alpha_i < \beta_i < \gamma_i \le n$ for all $i$.
Then $g$ is nontrivial in $F_n$.
\label{lem_freeword}
\end{lemma}

\begin{proof}
A word in the free group is trivial if and only if it can be reduced to the empty word by cancelling adjacent $x_\alpha x_\alpha^{-1}$ or $x_\alpha^{-1}x_\alpha$ pairs.
Within a given iterated commutator $$[x_{\alpha_i}, [ x_{\beta_i}, x_{\gamma_i}]]  = x_{\alpha_i}(x_{\beta_i}x_{\gamma_i}x_{\beta_i}^{-1}x_{\gamma_i}^{-1})x_{\alpha_i}^{-1}(x_{\gamma_i}x_{\beta_i}x_{\gamma_i}^{-1}x_{\beta_i}^{-1})$$ there are no such pairs since $1 \le \alpha_i < \beta_i < \gamma_i \le n$.
Between adjacent commutators, we can have at most two pairs of elements that cancel.
Indeed, each iterated commutator begins with the product $x_{\alpha_i}x_{\beta_i}x_{\gamma_i}$ of three distinct generators but ends with the expression $x_{\beta_i}x_{\gamma_i}^{-1}x_{\beta_i}^{-1}$ using only two distinct generators.
It follows that whatever is between those two expressions in each iterated commutator will never be cancelled in the product $g$ of iterated commutators.
Hence the word $g$ cannot be reduced to the empty word.
\end{proof}

\begin{lemma}
Let $F_n$ and $g$ be as in Lemma \ref{lem_freeword} and let $S$ be a subcollection of the collection of generators $x_1, \cdots, x_n$.
Let $N(S)$ be the smallest normal subgroup of $F_n$ containing $S$, and $G$ be the quotient $F_n / N(S)$.
Then the image of $g$ under the quotient map is trivial if and only if at least one element of each triple $x_{\alpha_i},x_{\beta_i},x_{\gamma_i}$ is contained in $S$.
\label{lem_clauselinkage}
\end{lemma}
\begin{proof}
Denote the image of $g$ under the quotient map by $\bar g$.
First suppose that at least one element from each triple $x_{\alpha_i},x_{\beta_i},x_{\gamma_i}$ is contained in $S$.
Then the image of each iterated commutator $[x_{\alpha_i}, [ x_{\beta_i}, x_{\gamma_i}]]$ is trivial in $G$, and so $\bar g$ is trivial in $G$ as the product of these commutators.

Now we prove that if some triple contains no element of $S$, then $g$ must be nontrivial.
For each $i$ such that one of $x_{\alpha_i},x_{\beta_i},x_{\gamma_i}$ is in $S$, the image of $[x_{\alpha_i}, [ x_{\beta_i}, x_{\gamma_i}]]$ is trivial in $G$, so we can remove $[x_{\alpha_i}, [ x_{\beta_i}, x_{\gamma_i}]]$  from $g$ without changing its image.
Due to this, we may assume that no triple of $g$ contains an element of $S$.
Next, observe that $G$ is a free group on $n - |S|$ generators.
Indeed, define a map from $F_n$ to $F_{n-|S|}$ by crossing out all appearances of elements of $S$ and their inverses.
It is simple to check that this is a surjective group homomorphism with kernel $N(S)$, and image $G$.
The generators of $G$ can be taken to be the set of $x_\alpha$ not contained in $S$.
The word $g$ is a sequence of iterated commutators in these generators, so we can apply Lemma \ref{lem_freeword} to see that the image of $g$ is nontrivial.
\end{proof}

We now begin the construction.

\begin{proof}[Proof of Theorem \ref{thm_splitting}]

Let $F$ be a 3-\textsc{SAT} formula with $n$ variables and $m$ clauses.
We construct a link $L^F$ that has splitting number $n$ if and only if $F$ has a solution.
First, add in additional variables $x_{n+1} \dots x_{2n}$ and clauses $x_i \vee \neg x_i \vee x_{n+i}$ and $x_i \vee \neg x_i \vee \neg x_{n+i}$ for $1 \le i \le n$.
These new clauses and variables do not change the satisfiability of the 3-\text{SAT} formula.
Indeed, all new clauses are satisfied by any TRUE/FALSE variable assignment.
The new formula $F'$ has $n' = 2n$ variables and $m' = m+2n$ clauses.

To construct the link, begin with a $2n'$-component unlink, and label the components $x_1, \cdots, x_{n'}, \neg x_1, \cdots , \neg x_{n'}$.
The complement of this unlink has fundamental group $F_{2n'}$, and Wirtinger generators $$x_1, \cdots, x_{n'}, \neg x_1, \cdots , \neg x_{n'}$$ each corresponding to a loop going once around the strand of the link component with the respective label.
Suppose $F'$ has clauses $c_i =  z_{\alpha_i} \vee z_{\beta_i} \vee z_{\gamma_i}$ for $1 \le i \le m'$ and $1 \le \alpha_i,\beta_i,\gamma_i \le n'$, where $z_{\alpha_i}$ is either $x_{\alpha_i}$ or $\neg x_{\alpha_i}$ appearing in $c_i$.
Then we add a single component $P$ corresponding to the word
$$[z_{\alpha_1}, [ z_{\beta_1}, z_{\gamma_1}]][z_{\alpha_2}, [ z_{\beta_2}, z_{\gamma_2}]] \cdots [z_{\alpha_{m'}}, [ z_{\beta_{m'}}, z_{\gamma_{m'}}]]$$
in the fundamental group.
Once the above product of the iterated commutators is written as a word in $x_\alpha^{\pm 1}$ and $(\neg x_\alpha)^{\pm 1}$, $P$ is drawn using the same method as in the proof of Theorem \ref{thm_unlinksublink} with respect to the link strands/components labeled by the generators.
Note however that the strands do not belong to Hopf sublinks anymore as in Theorem \ref{thm_unlinksublink}, but rather to unlinked components.
See Figure \ref{fig_clauselink2} for a simple example.

\begin{figure}[ht]
\centering
\includegraphics[height=3in]{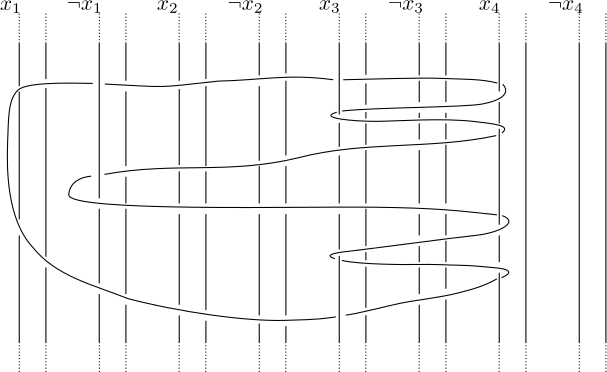}
\caption{A link representing the commutator $[\neg x_1,[x_3,x_4]]$, corresponding to the clause $ \neg x_1 \vee x_3 \vee x_4$.}
\label{fig_clauselink2}
\end{figure}

We claim $P$ is not split from the collection of other components.
In fact, suppose $S$ is some subset of $x_1, \cdots, x_{n'}, \neg x_1, \cdots , \neg x_{n'}$ and we delete the components with labels corresponding to elements of $S$ to get a link $L_S$.
Then by Lemma \ref{lem_clauselinkage}, $P$ is split from the other components in $L_S$ if and only if each triple $\alpha_i,\beta_i,\gamma_i$ contains some element of $S$.

The next step is to construct an intermediate link $L^d$ by replacing each of the components $x_1, ..., x_{n'}, \neg x_1, ..., \neg x_{n'}$ (but not the combined clause component $P$) with an untwisted Whitehead double.
We call these Whitehead doubles the \emph{variable components} of $L^d$, in contrast to the clause component $P$, and continue to refer to them by the same labels.
Note that as a Whitehead double, each variable component can be made into an unknot unlinked from the other components by a single crossing change (or unlinking move, in the diagram free terminology).
We again call this crossing change an \emph{unclasping move}.
Thus, the splitting number of $L^d$ is always one, although splitting the link in this way will in general not split $P$ from the other components.

We will show that in $L^d$, the component $P$ can be split from all other components in $n$ unlinking moves if and only if there is a solution for $F$. More specifically, we claim that if $P$ can be made split from the rest of $L^d$ by $n$ unlinking moves, then there is a set of $n$ unclasping moves for variable components that also splits $P$ from the rest of $L^d$, and that this set of $n$ unclasping moves corresponds to an assignment satisfying $F$. Here unclasping a component labelled $z_{\beta_k}$ means assigning TRUE to $z_{\beta_k}$ in the variable assignment for $F$. Note that such an assignment also corresponds to a partial assignment of variables of $F'$, but leaves $n$ of the variables of $F'$ unassigned. However, the construction of $F'$ ensures that if such an assignment satisfies all clauses of $F$, then any arbitrary choice of TRUE/FALSE for the last $n$ variables of $F'$ will satisfy all clauses of $F'$.  That is, every clause of $F'$ is assured to be true by the assignment given to one of the first $n$ variables.  It follows that in the group element $g$ corresponding to the component $P$, every iterated commutator will collapse to the identity in the group quotient determined by the partial variable assignment as per Lemmas \ref{lem_freeword} and \ref{lem_clauselinkage}.

Suppose $P$ can be split from the other components of $L^d$ in $n$ unlinking moves.
Denote this set of unlinking moves by $C$. We now apply Lemma \ref{lem_unlinking} to every variable component of $L^d$ such that there is a move in $C$ involving this component. The chosen component serves as $K$ in the lemma.
Then using Lemma \ref{lem_unlinking} (a) and (c) for such components we can replace any move in $C$ involving a variable component with an unlinking move unclasping this component. After all possible such replacements in $C$, the new set of unlinking moves $C'$ consists of some moves unclasping variable components of $L^d$ and some moves that only involve $P$. We will see that the latter moves cannot appear in $C'$.

Note that in $C'$, the moves involving variable components do not yet determine an assignment of variables of $F$ in the right way. In particular, a variable might be assigned both TRUE and FALSE if there is an unclasping move for both $x_i$ and for $\neg x_i$. A variable of $F$ may also be not assigned anything at all. In what follows, the extra clauses that we added to $F$ to obtain $F'$ will be used to determine a variable assignment for $F$ fully. Therefore our assignment for $F$ will not come from the current sequence $C'$. Rather, we will modify $C'$.

Suppose a link component that corresponds to $x_{n+i}$ or $\neg x_{n+i}$ is unclasped for $1 \le i \le n$ by an unlinking move $v$ in $C'$.
We might have one of two scenarios.

In the first scenario, there is no unclasping move for the link components labeled either $x_i$ or $\neg x_i$.  In this case $v$ from $C'$ can be replaced by either such unclasping move. Indeed, in Lemma \ref{lem_clauselinkage}, replacing $x_{n+i}$ or $\neg x_{n+i}$ with $x_i$ or $\neg x_i$ in the set $S$ still yields a trivial image of $g$, since at least one element of every triple $x_{\alpha_i},x_{\beta_i},x_{\gamma_i}$ is still contained in $S$ after the replacement. Therefore, such a replacement of $v$ in $C'$ yields the sequence that still splits $P$.

In the second scenario, there is already another unclasping move $u$ in $C'$ for either $x_i$ or $\neg x_i$. This means that either (1) we applied Lemma 4 to $x_i$ or $\lnot x_i$ when we constructed $C'$ from $C$, to obtain this unclasping move; or (2) the move $u$ was one of the original linking moves in $C$; or (3) we already replaced another unclasping move of $x_{n+i}$ or $\neg x_{n+i}$ in $C'$ with $u$ as in the first scenario.  In all cases, the move $c$ can be deleted from $C'$ without changing whether $F$ is satisfied.

Note that if we already have an unclasping move $u$ for $x_i$ or $\neg x_i$ in $C'$, then $u$ yields a part of a variable assignment. All clauses of $F'$ involving the component $x_{n+i}$ or $\neg x_{n+i}$ are already satisfied by assigning $x_i$ to either TRUE or FALSE, so the TRUE/FALSE assignment of $x_{n+i}$ corresponding to the unclasping move $c$ is redundant. Indeed, no clauses of the original formula $F$ involve $x_{n+i}$ or $\neg x_{n+i}$, and the added clauses $x_i \vee \neg x_i \vee x_{n+i}$ and $x_i \vee \neg x_i \vee \neg x_{n+i}$ will both still be satisfied after such a replacement or deletion.

After all modifications, denote the resulting sequence by $C''$. As explained above, $C''$ splits $P$ from the rest of $L^d$. At the end of this process, we can assume that all unclasping moves in $C''$ are of link components corresponding to $x_i$ or $\neg x_i$ for $1 \le i \le n$.
The clauses $x_i \vee \neg x_i \vee x_{n+i}$ then ensure that either $x_i$ or $\neg x_i$ is unclasped for all $1 \le i \le n$.
Since there are at most $n$ unclaspings, $x_i$ and $\neg x_i$ cannot both be unclasped for any $i$.
The choice of which variable components to unclasp therefore corresponds to a TRUE/FALSE assignment to $x_1 \dots x_n$, and these unclaspings will result in the clause component being split if and only each clause of $F$ is satisfied.

It follows that there must be precisely $n$ unclasping moves in $C''$.
Since $C''$ is a set of $n$ unlinking moves, all unlinking moves must therefore be unclasping moves, and there can be no moves involving $P$.
The set of unlinking moves $C''$ therefore corresponds to a variable assignment in $F$, and the clause component can be split from the other components in $n$ unlinking moves if and only if there is a solution for $F$.
Note again that the splitting number of $L^d$ is always 1, since unclasping one side of any variable produces a split link.

To obtain the final link $L_F$, we add one more component to $L^d$.
Consider the curve $E$ drawn in grey in Figure \ref{fig_hopfdoublechain}.
This curve has linking number 1 with every link component other than the combined clause component, and we can choose this curve to have no crossings with the combined clause component.
Thicken the curve $E$, and let $E'$ be a curve going around the boundary torus of the thickened curve $n+1$ times longitudinally and once meridianally.
Then $E'$ links with every non-clause component $n+1$ times.
\begin{figure}[h]
\centering
\includegraphics[height=1.3in]{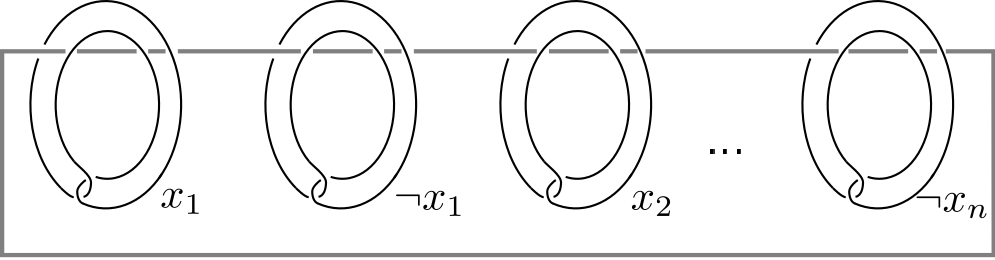}
\caption{A curve $E$ in thick grey linking every non-clause link component.}
\label{fig_hopfdoublechain}
\end{figure}

Let $L_F$ be the union of the variable components, the combined clause component, and $E'$.
The component $E'$ cannot be unlinked from any of the variable components in fewer than $n+1$ moves, and the
variable components cannot be unlinked from $E'$ in fewer than $n+1$ moves. So the only way to get a split link from $L_F$ in at most $n$ unlinking moves is to split the combined clause component from the remaining components.
As argued above for the link without $E'$, this is possible if and only if $F$ has a solution. This is just as true after adding the
component $E'$.
Thus, $L_F$ can be split in $n$ unlinking moves if and only if there is a solution to the formula $F$.

The size of the diagram of $L_F$ is polynomial in the size of the input for $F$.
Indeed, before adding $E'$ the number of crossings is less than the number of crossings involved in the construction for \textsc{unlinking number}, which was shown to be polynomial in the size of the input in Theorem \ref{thm_unlinking}.
$K$ has $n$ self-crossings, and $8n(n+1)$ crossings with other components, so adding the component $E'$ adds polynomially many more crossings.
Therefore, we have a polynomial reduction and the theorem follows.
\end{proof}

\begin{rem}
The method of combining all clause link components into a single component can also be applied to the \textsc{unlink as a sublink}, \textsc{unlinking number}, and \textsc{alternating sublink}, though we use different methods there.
While combining all components allows the number of components in the link corresponding to a given 3-SAT formula to be reduced, the proofs become more complicated.
\end{rem}

\section{Alternating Sublinks}
\label{sec_alternatingsublink}

We next consider another variation of the \textsc{sublink problem}.

\vspace{0.14in}

\textsc{alternating sublink}: Given a diagram of a link $L$ and a positive integer $k$, does $L$ have a $k$-component sublink that is alternating?

\vspace{0.14in}

Note that unlike the \textsc{sublink problem} as proposed by Lackenby and the special case \textsc{unlink as a sublink} analyzed in Theorem \ref{thm_unlinksublink}, the problem above asks not for a specific sublink, but a sublink with a given property.

\begin{thm}
\textsc{alternating sublink} is NP--hard.
\label{thm_altsub}
\end{thm}

We need the following lemma:

\begin{lemma}
Let $L_s$ be a link that is not the unlink, and $L_s'$ the result of replacing every component of $L_s$ with a Whitehead double.
Then $L_s'$ is not alternating.
\label{lem_whitehead}
\end{lemma}
\begin{proof}
The link $L_s$ is split if and only if $L_s'$ is split.
A splitting sphere for $L_s$ also splits $L_s'$, while a splitting sphere for $L_s'$ can be isotoped to obtain a splitting sphere disjoint from all satellite tori, and is therefore a splitting sphere for $L_s$.

We will assume $L_s'$ is alternating and arrive at a contradiction.
Consider a reduced alternating diagram $D$ of $L_s'$.
If $L_s'$ is split, $L_s'$ can be seen to be split in $D$ as was proved by Menasco \cite{Menasco1}.
Moreover, every connected piece of $D$ is alternating, and at least one of such pieces, say $P$, is not an unknot, since $L_s'$ is not an unlink.
The piece $P$ is alternating and has a connected diagram by assumption, and hence $P$ is non-split as a link in $S^3$.
We claim that $P$ is then prime as a link in $S^3$.
Menasco's results \cite{Menasco1} imply that a prime link complement cannot contain an incompressible non-boundary parallel torus, but there is a satellite torus in the complement $S^3-P$, a contradiction.

Let us now prove the claim that $P$ is prime.
Suppose $\Sigma$ is an embedded 2-sphere intersecting $P$ in two points.
Let $K'$ be the link component of $P$ that $\Sigma$ intersects, and let $T$ be the companion torus for $K'$.
Let $K$ be the component of $L_s$ of which $K'$ is a Whitehead double, so $T$ can be thought of as the boundary of a regular neighborhood of $K$.
See Figure \ref{fig_lemmapica}.

\begin{figure}[ht]
\centering
    \begin{subfigure}[b]{0.5\textwidth}
        \centering
        \includegraphics[height=1.9in]{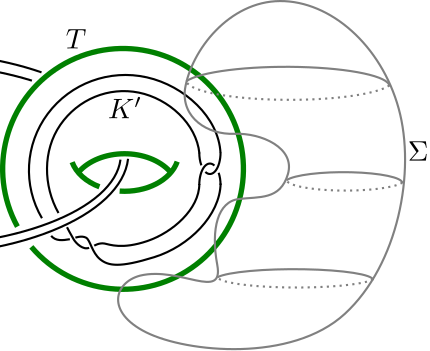}
        \caption{}
        \label{fig_lemmapica}
    \end{subfigure}%
~~~~~
    \begin{subfigure}[b]{0.4\textwidth}
        \centering
        \raisebox{0.17\height}{\includegraphics[height=1.5in]{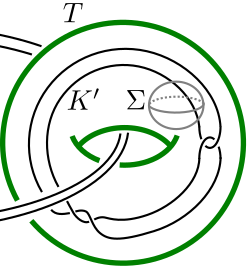}}
        \caption{}
        \label{fig_lemmapicb}
    \end{subfigure}
\caption{The link component $K'$ (black) in it's companion torus $T$ (thick green), and the sphere $\Sigma$ (grey) intersecting $K'$ in two points.
Fragments of other components of $L$ are also depicted in black.
}
\label{fig_lemmapic}

\end{figure}

If the 2-sphere $\Sigma$ intersects $T$ in an essential meridianal curve, there must be at least two such curves of intersection.
Therefore, the number of intersections of $\Sigma$ and the component $K'$ is at least 4, contradicting the definition of $\Sigma$.
If $\Sigma$ intersects $T$ in an inessential curve, we can isotope $\Sigma$ to remove this intersection, since $P$ is not split.
Since $\Sigma$ intersects $K'$ by assumption, $\Sigma$ lies entirely on the same side of $T$ as $K'$, and any other components of $P$ lie on the other side of $T$, since $T$ is a companion torus for $K$, as shown in Figure \ref{fig_lemmapicb}.
Therefore, the connected component of the manifold $S^3 - \Sigma$ that does not contain $T$ intersects $K'$ and no other components of $P$.
We claim this intersection with $K'$ is an unknotted arc.

Let $V$ be the solid torus bounded by $T$ and containing $K'$.
Consider a homeomorphism $h$ of $V$ such that the resulting solid torus $h(V)$ is embedded and unknotted in $S^3$, and $K'$ inside $V'$ is an untwisted whitehead double of an unknot.
The map $h$ can be visualized, for example, as cutting $V$ along a meridianal disk, unknotting and untwisting it if necessary, and regluing along the same disk again.
Then there is a 2-sphere $h(\Sigma)$ inside $h(V)$ intersecting $h(K')$ in two points.
The sphere $h(\Sigma)$ splits $S^3$ into two 3-balls.
Since $h(K')$ is the Whitehead double of an unknot, and is therefore unknotted itself, it must intersect both 3-balls in unknotted arcs.
On the other hand, $h$ did not change the interior of $V$, and therefore pulling back by $h$ shows that $\Sigma$ bounds an unknotted arc in $V$ as well.
We have shown that any 2-sphere intersecting $\Sigma$ in two points bounds an unknotted arc on one side, and hence $P$ is prime.
\end{proof}

\begin{proof}[Proof of Theorem \ref{thm_altsub}]
Given a 3-SAT formula $F$ with $n$ variables and $m$ components, we use the same construction for a link $L^F_*$ as for \textsc{unlinking number} in Theorem \ref{thm_unlinking}.
Recall that each component of the link $L^F_*$ is the Whitehead double of a component of the respective link $L^F$ used for \textsc{unlink as a sublink} in the proof of Theorem \ref{thm_unlinksublink}.
Let $L_s$ be an $(n+m)$-component sublink of $L^F$, before replacing all components with their Whitehead doubles, and $L_s'$ be the respective sublink of $L^F_*$ after replacing the components with Whitehead doubles.

We proved in Theorems \ref{thm_unlinksublink} and \ref{thm_unlinking} that each of the sublinks $L_s$ and $L_s'$ is the unlink if and only if variable components of the sublink correspond to FALSE values in a solution of the 3--\textsc{SAT} problem $F$.
If $L_s'$ does not correspond to a solution of $F$, the sublink $L_s$ is not an unlink, and hence $L_s'$ cannot be alternating by Lemma \ref{lem_whitehead}.
On the other hand, if $L_s'$ is an unlink, it is trivially alternating.
Therefore, $L_s'$ is alternating if and only if it corresponds to a solution of $F$.
This Karp reduction is the same as the one used for \textsc{unlinking number}, so it is polynomial in the size of input for $F$.
It follows that \textsc{alternating sublink} is NP--hard.
\end{proof}

\section{Bound for Reidemeister moves for an arbitrary link}
\label{sec_reidbound}

We next consider the problem of bounding the number of Reidemeister moves needed to pass from one diagram of a link to another.

\vspace{0.14in}
\textsc{bound for reidemeister moves for an arbitrary link}: Given two link diagrams and an integer $k$, are the link diagrams related by a sequence of at most $k$ Reidemeister moves?
\vspace{0.14in}

 We will use the following restricted problem to prove NP-hardness of \textsc{bound for reidemeister moves for an arbitrary link}.

\vspace{0.14in}
\textsc{bound for reidemeister moves for an unlink/unknot}: Given two diagrams of an unlink/unknot and an integer $k$, are the diagrams related by a sequence of at most $k$ Reidemeister moves?
\vspace{0.14in}

Another variation of the problem is where only one diagram is assumed to be the unlink/unknot diagram and the knot type of the other diagram is unknown.
We will call this \textsc{bound for reidemeister moves proving homeomorphism to an unlink/unknot}.
Note that we do not assume that the unlink/unknot diagram is trivial anywhere above, i.e.
that it consists of unknotted disjoint circles.
If one of the diagrams is replaced by a trivial unlink/unknot diagram in \textsc{bound for reidemeister moves for an unlink/unknot}, we will call the problem \textsc{bound for reidemeister moves leading to a trivial unlink/unknot diagram}.

\subsection{Reidemeister moves problems and upper bounds on complexity.}

The problem \textsc{bound for reidemeister moves for an arbitrary link} might not be in NP in general.
The obvious certificate consisting of a sequence of $k$ diagrams and Reidemeister moves connecting them has size at least linear in $k$.
But the size of the input is $O(n_1+n_2+\log k)$, where $n_1$ and $n_2$ are the number of crossings in the two given link diagrams, and $\log k$ is the number of bits necessary to input the integer $k$.
Hence the length of the shortest sequence of moves connecting two diagrams might be exponential in the size of the input.

However, in special cases bounds are known for the number of Reidemeister moves needed to reduce a diagram, thus reducing the dependence of the complexity on $k$.
The first bounds were found by Hass and Lagarias \cite{HassLagarias1} for the number of Reidemeister moves needed to reduce a diagram of an unknot to the trivial unknot/unlink diagram.
Later Hayashi \cite{Hayashi1} found a bound for the number of Reidemeister moves needed to get a split diagram of a split link.
More recently, bounds have been found that are polynomial in the number of crossings.
Hass, Lagarias, and Pippenger in \cite{HassLagariasPippenger1} and Lackenby in \cite{Lackenby2} provide polynomial bounds on the number of Reidemeister moves needed to change any diagram of an unknot to its trivial diagram, and to split a diagram of a split link respectively.

These results quickly lead to a polynomial bound on the number of Reidemeister moves required to change a diagram of an $n$ component unlink to the diagram with no crossings.
Such a sequence consists of $n-1$ sequences of moves splitting the components from each other in the diagram, followed by $n$ sequences of moves each changing one of the components into the 0--crossing diagram of the unknot.
So the sequence is made up of $2n-1$ smaller sequences, and each is polynomial in the input by \cite{HassLagariasPippenger1,Lackenby2}.
Therefore, the total number of Reidemeister moves needed is polynomial as well.

Given two nontrivial diagrams of the unlink, there is a polynomial bound on the number of Reidemeister moves needed to obtain one from another consisting of a sequence changing the first diagram to the 0--crossing diagram followed by a sequence changing the 0--crossing diagram to the second diagram.
If $k$ is smaller than this bound, then there is a certificate polynomial in the input, and if $k$ is larger, then \cite{HassLagariasPippenger1,Lackenby2} prove the existence of such a sequence.
It follows that \textsc{bound for reidemeister moves for an unlink/unknot} is in NP.
The same certificates show that \textsc{bound for reidemeister moves proving homeomorphism to an unlink/unknot} and \textsc{bound for reidemeister moves leading to a standard unlink/unknot diagram} are in NP as well.

Additionally, \cite{Lackenby3} refers to a paper in progress \cite{Lackenby4} showing that for any given link type $L$ there is a polynomial $p_K$ such that any
two diagrams of $L$ with crossing numbers $n$ and $n'$ differ by a sequence of at most $p_K(n) + p_K(n')$ Reidemeister moves.
This suggests that the upper bound problem for Reidemeister move equivalence is in NP whenever it is restricted to a fixed link type $L$.
Note however that a polynomial independent of the link type would be needed for this to imply that \textsc{bound for reidemeister moves for an arbitrary link} is in NP.
In fact, such a result would imply that link equivalence is in NP.

\subsection{Reidemeister moves problems and lower bounds on complexity.}

We establish the first lower complexity bound for Reidemeister moves problems, showing that \textsc{bound for reidemeister moves for an unlink} is NP--hard.
Since this problem is in NP, we have:

\begin{thm}
\textsc{bound for reidemeister moves for an unlink} is NP--complete.
\label{thm_reid}
\end{thm}

Since this is a subproblem of the more general case, it immediately implies:
\begin{corr}
\textsc{bound for reidemeister moves for an arbitrary link} is NP--hard.
\end{corr}

\textsc{bound for reidemeister moves proving homeomorphism to an unlink/unknot} is NP--hard as it is a harder problem than \textsc{bound for reidemeister moves for an unlink}, as strictly less information is given.
It is also in NP as discussed earlier, and is therefore NP--complete as well.
It is not yet known whether \textsc{bound for reidemeister moves leading to a standard unlink/unknot diagram} is NP--hard.

As \textsc{bound for reidemeister moves for an unknot} is a restriction of the similar problem for an unlink, it is possibly easier for an unknot than an unlink.
Therefore, it is not obvious whether \textsc{Bound for Reidemeister Moves for an Unknot} is NP-hard, as it is unclear how to reduce the argument of Theorem 15 to single component links.

\begin{question}
Is \textsc{bound for reidemeister moves for an unknot} NP--hard?
\end{question}

Soon after the first version of the preprint of this paper appeared online, the question was resolved positively in \cite{MRST}.

We conclude with the proof of Theorem \ref{thm_reid}.

\begin{proof}[Proof of Theorem \ref{thm_reid}]
We will use the \textsc{planar hamiltonian path} problem with a small modification.
Specifically, we assume that the graph $\Gamma$ that we construct has at least two vertices of degree 1.
Since the authors did not find a direct reference to this problem in the literature (with or without the modification), we modify the known proof for \textsc{planar hamiltonian cycle}, proved to be NP--complete by Garey, Johnson, and Tarjan \cite{GareyJohnsonTarjan1}, to show that our version of the \textsc{planar hamiltonian path} problem is also NP--complete.
In their construction, in every graph there exist edges that must be a part of any Hamiltonian cycle (if there are no Hamiltonian cycles in a graph, this is true vacuously).
Fix a graph $\Gamma'$.
We choose one such edge $e$ in $\Gamma'$ and replace it with two degree one vertices, each connected to one of the endpoints of $e$ as in Figure \ref{edgerep}.
Let $\Gamma$ denote the resulting graph.
The new graph $\Gamma$ contains a Hamiltonian path if and only if the original graph $\Gamma'$ contains a Hamiltonian cycle.
We may assume that $\Gamma$ has $n$ vertices and $m$ edges, with $m \ge n-1$ (if $m < n-1$, then there is no Hamiltonian path).

\begin{figure}[ht]
\centering
    \begin{subfigure}[b]{0.4\textwidth}
        \centering
        \includegraphics[height=1.0in]{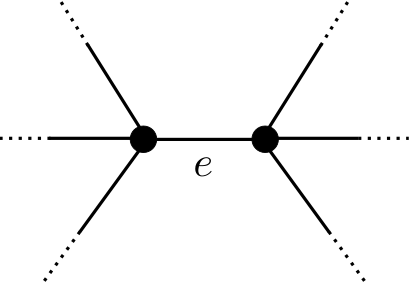}
        \caption{}
    \end{subfigure}%
~~~~~
    \begin{subfigure}[b]{0.5\textwidth}
        \centering
        \includegraphics[height=1.0in]{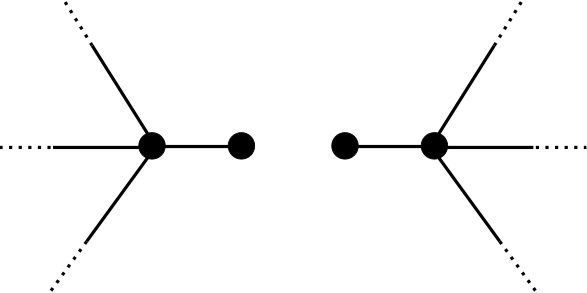}
        \caption{}
    \end{subfigure}
\caption{Given an edge $e$ that lies in any Hamiltonian cycle (\textsc{a}), replace $e$ with two vertices connected by edges to the vertices of $e$ as in (\textsc{b}) to obtain the graph $\Gamma$.}
\label{edgerep}

\end{figure}

We now reduce the modified Hamiltonian path problem to \textsc{bound for reidemeister moves for an unlink}.
We construct two link diagrams $D_1, D_2$ of an unlink from the graph $\Gamma$ such that $D_1$ and $D_2$ are related in at most $2(m-n+1)$ Reidemeister moves if and only if $\Gamma$ has a Hamiltonian path.
For the first diagram $D_1$, draw an unknot for each vertex of $\Gamma$.
We refer to these as the vertex components of $D_1$.
For each edge of $\Gamma$, draw an unknot that passes under the unknots from the adjacent vertices without linking with them as in Figure \ref{fig_reidboundb}.
We call these components edge components of $D_1$.
Let $D_2$ be the union of a chain of $2n-1$ unlinked components and $m-n+1$ unknots not touching anything else as in Figure \ref{fig_reidboundc}.

\begin{figure}[ht]
\centering
    \begin{subfigure}[b]{0.4\textwidth}
        \centering
        \includegraphics[height=1.5in]{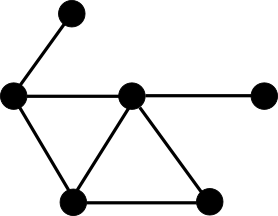}
        \caption{$\Gamma$}
    \end{subfigure}%
    ~
    \begin{subfigure}[b]{0.4\textwidth}
        \centering
        \includegraphics[height=1.5in]{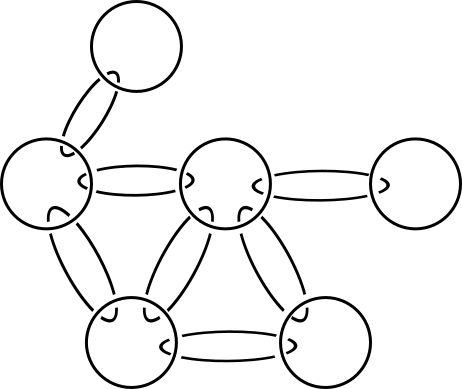}
        \caption{$D_1$}
        \label{fig_reidboundb}
    \end{subfigure}

        \begin{subfigure}[b]{0.8\textwidth}
        \centering
        \includegraphics[height=1in]{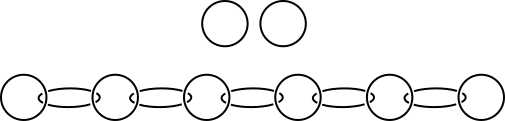}
        \caption{$D_2$}
        \label{fig_reidboundc}
    \end{subfigure}
\caption{We reduce the problem of finding a Hamiltonian path in $\Gamma$ to the problem of finding if $D_1$ and $D_2$ are related in at most $2(m-n+1)=4$ Reidemeister moves.}
\label{fig_reidbound}
\end{figure}

By counting the number of crossings in $D_1$ and $D_2$, we see that the only way to get from $D_1$ to $D_2$ in the required number of moves is to do solely Reidemeister moves of type II reducing the crossing number.
Indeed, starting with $D_1$ and performing $2(m-n+1)$ type II Reidemeister moves reducing the crossing number will result in a link diagram with the same crossing number as $D_2$.
If any other types of Reidemeister moves are performed instead, the final crossing number will be larger than that $D_2$.
Thus, every move must disconnect an edge component of $D_2$ from a vertex component.

The sequence of Reidemeister moves leads to the diagram $D_2$, where there is a connected chain of unknots.
Keeping track of which components were edge components and which were vertex components of $D_1$, we see that the resulting chain must alternate between edge and vertex components, and contain either $n$ vertex components and $n-1$ edge components or $n-1$ vertex components and $n$ edge components.
If a degree 1 vertex of $\Gamma$ is part of the chain, it can only lie at one of the endpoints of the chain.
Since there are two degree 1 vertices, this forces the case of $n$ vertex components and $n-1$ edge components.
Therefore, the chain in $D_2$  must correspond to a path of $\Gamma$ containing every vertex of the graph, i.e.
a Hamiltonian path.

Note that the first diagram has $4m$ crossings, and the second diagram has $4n-2$ crossings, hence both link diagrams have size polynomial in the size of the input graph.
Therefore, the reduction is polynomial.\end{proof}

\section{Acknowledgments}
Anastasiia
Tsvietkova acknowledges support from
NSF DMS-1664425 (previously 1406588) and NSF DMS-2005496 grants, and Insitute of
Advanced Study (under DMS-1926686 grant).
Both authors were supported by Okinawa Institute of Science and Technology funding. We are thankful to the referee for careful reading of the paper and for useful comments.

\

\

\

\bibliographystyle{amsplain}
\bibliography{./complexity}{}

\providecommand{\bysame}{\leavevmode\hbox to3em{\hrulefill}\thinspace}
\providecommand{\MR}{\relax\ifhmode\unskip\space\fi MR }
\providecommand{\MRhref}[2]{%
  \href{http://www.ams.org/mathscinet-getitem?mr=#1}{#2}
}
\providecommand{\href}[2]{#2}
\begin{thebibliography}{10}

\bibitem{AgolHassThurston1}
Ian Agol, Joel Hass, and William Thurston, \emph{The computational complexity
  of knot genus and spanning area}, Trans. Amer. Math. Soc. \textbf{358}
  (2006), no.~9, 3821--3850.

\bibitem{MRST}
Eric Sedgwick Martin~Tancer Arnaud~de Mesmay, Yo'av~Rieck, \emph{The unbearable
  hardness of unknotting}, arxiv.org/abs/1810.03502.

\bibitem{CowardLackenby2}
Alexander Coward and Mark Lackenby, \emph{Unknotting genus one knots}, Comment.
  Math. Helv. \textbf{86} (2011), no.~2, 383--399.

\bibitem{CowardLackenby1}
\bysame, \emph{An upper bound on {R}eidemeister moves}, American Journal of
  Mathematics \textbf{136} (2014), no.~4, 1023--1066.

\bibitem{GareyJohnson1}
Michael~R. Garey and David~S. Johnson, \emph{{Computers and Intractability: A
  guide to the Theory of NP-Completeness}}, W.H. Freeman, 1979.

\bibitem{GareyJohnsonTarjan1}
Michael~R. Garey, David~S. Johnson, and Robert Endre~Tarjan, \emph{The planar
  {H}amiltonian circuit problem is {NP}-complete}, {SIAM} J Comput. \textbf{5}
  (1976), no.~4, 704--714.

\bibitem{Goldreich}
Oded Goldreich, \emph{{P, NP, and NP-completeness: The Basics of Computational
  Complexity}}, Cambridge University Press, 2010.

\bibitem{HassLagarias1}
Joel Hass and Jeffrey Lagarias, \emph{The number of {R}eidemeister moves needed
  for unknotting}, J. Amer. Math. Soc. \textbf{14} (2001), no.~2, 399--428.

\bibitem{HassLagariasPippenger1}
Joel Hass, Jeffrey Lagarias, and Nicholas Pippenger, \emph{The computational
  complexity of knot and link problems}, {JACM} \textbf{46} (1999), no.~2,
  185--211.

\bibitem{Hayashi1}
Hayashi, \emph{The number of {R}eidemeister moves for splitting a link}, Math.
  Ann \textbf{332} (2005), no.~2, 239--252.

\bibitem{Lackenby3}
Mark Lackenby, \emph{Elementary knot theory}, arXiv:1604.03778.

\bibitem{Lackenby4}
\bysame, \emph{A polynomial upper bound on {R}eidemeister moves for each link
  type}, In preparation.

\bibitem{Lackenby2}
\bysame, \emph{A polynomial upper bound on {R}eidemeister moves}, Ann. of Math.
  \textbf{182} (2015), no.~2, 491--564.

\bibitem{Lackenby1}
\bysame, \emph{Some conditionally hard problems on links and 3-manifolds},
  Discrete Comput. Geom. \textbf{58} (2017), no.~3, 580--595, arXiv:1602.08427.

\bibitem{Lackenby5}
\bysame, \emph{Links with splitting number one}, arXiv:1808.05495, 2018.

\bibitem{Menasco1}
William Menasco, \emph{Closed incompressible surfaces in alternating knot and
  link complements}, Topology \textbf{23} (1984), no.~1, 37--44.

\bibitem{Reidemeister1}
Kurt Reidemeister, \emph{Knotten und gruppen.}, Abh. Math. Sem. Univ. Hamburg
  \textbf{5} (1927), 7--23.

\bibitem{Rolfsen1}
Dale Rolfsen, \emph{Knots and links}, ch.~3F, p.~67, American Mathematical
  Society, 1976.

\end{thebibliography}

\

Dale Koenig \\
Rapyuta Robotics \\
Tokyo-to Koto-ku \\
Furuishiba 2-chome 14-9-201, Japan \\
dale.koenig@rapyuta-robotics.com \\

Anastasiia Tsvietkova\\
Department of Mathematics and Computer Science\\
Rutgers University-Newark \\
101 Warren Street, Newark, NJ  07102, USA\\
a.tsviet@rutgers.edu

\end{document}